\documentclass[11pt]{amsart}
\usepackage{amsxtra}
\usepackage[all]{xy}

\theoremstyle{plain}

\newcommand{\cleqn}{\setcounter{equation}{0}}
\newcommand{\clth}{\setcounter{theorem}{0}}
\newcommand {\sectionnew}[1]{\section{#1}\cleqn\clth}

\newtheorem{theorem}{Theorem}[section]
\newtheorem{lemma}[theorem]{Lemma}
\newtheorem{definition-lemma}[theorem]{Definition-Lemma}
\newtheorem{proposition}[theorem]{Proposition}
\newtheorem{corollary}[theorem]{Corollary}

\newtheorem{example}[theorem]{Example}
\newtheorem{remark}[theorem]{Remark}

\newcommand \lb[1]{\label{#1}}

\newcommand{\st} [1]     {\scriptscriptstyle{{#1}}}
\newcommand{\rmap}       {\longrightarrow}

\newcommand{\Ker}        {{\mathrm {ker}}}

\newcommand{\Id}         {\mathrm{Id}}

\newcommand{\Grd}        {\mathcal{G}}

\newcommand{\SP} [1]     {{\left\langle {{#1}} \right\rangle}}

\newcommand{\LF}       {\mathrm{Lie}} 

\newcommand{\sour}        {\mathsf{s}}
\newcommand{\tar}         {{\mathsf{t}}}

\newcommand{\Lie}        {\mathcal L}

\newcommand{\reals}      {{\mathbb R}}

\begin{document}
\title[]
{Linear and multiplicative 2-forms}

\author[]{Henrique Bursztyn, Alejandro Cabrera, Cristi\'an Ortiz}
\address{Instituto de Matem\'atica Pura e Aplicada,
Estrada Dona Castorina 110, Rio de Janeiro, 22460-320, Brasil }
\email{henrique@impa.br, cabrera@impa.br, cortiz@impa.br}

\date{}

\maketitle

\begin{abstract}
We study the relationship between multiplicative 2-forms on Lie groupoids and
linear 2-forms on Lie algebroids, which leads to a new approach to the
infinitesimal description of multiplicative 2-forms and to the
integration of twisted Dirac manifolds.
\end{abstract}

\tableofcontents

\sectionnew{Introduction}\lb{intro}

The main purpose of this paper is to offer an  alternative viewpoint
to the study of multiplicative 2-forms on Lie groupoids and their
infinitesimal counterparts carried out in \cite{bcwz}. This study
turns out to be closely related to topics such as equivariant
cohomology and generalized moment map theories, see, e.g.,
\cite{BC,bcwz,xu}. A particularly important case is that of the
symplectic multiplicative 2-forms of \textit{symplectic groupoids}
\cite{CDW}, whose infinitesimal counterparts are Poisson structures.
As shown in \cite{bcwz}, infinitesimal versions of more general
multiplicative 2-forms include twisted Dirac structures in the sense
of \cite{SW}.

Let $\Grd$ be a Lie groupoid over $M$, with source and target maps
$\sour, \tar: \Grd \rmap M$, and multiplication $m:\Grd^{(2)}\rmap
\Grd$. Let $A$ be the Lie algebroid of $\Grd$, with Lie bracket
$[\cdot,\cdot]$ on $\Gamma(A)$ and anchor $\rho: A\rmap TM$. A
2-form $\omega \in \Omega^2(\Grd)$ is called \textit{multiplicative}
if
$$
m^*\omega = p_1^*\omega + p_2^*\omega,
$$
where $p_1, p_2: \Grd^{(2)}\rmap \Grd$ are the natural projections.
Given a closed $3$-form $\phi\in \Omega^3(M)$, we say that $\omega$
is \textit{relatively $\phi$-closed} if $\mathrm{d}\omega= \sour^*\phi -
\tar^*\phi$. The main result in \cite{bcwz} asserts that, if $\Grd$
is $\sour$-simply-connected, then there exists a one-to-one
correspondence between multiplicative 2-forms $\omega\in
\Omega^2(\Grd)$ and  vector-bundle maps $\sigma: A\rmap T^*M$
satisfying
\begin{align*}
&\SP{\sigma(u),\rho(v)}= - \SP{\sigma(v),\rho(u)}\\
&\sigma([u,v])=\Lie_{\rho(u)}\sigma(v)-i_{\rho(v)}\mathrm{d}\sigma(u) +
i_{\rho(v)}i_{\rho(u)}\phi,
\end{align*}
for all $u,v \in \Gamma(A)$. We refer to such maps $\sigma$ as
\textit{IM 2-forms} relative to $\phi$ (\textit{IM} stands for
\textit{infinitesimal multiplicative}). If $L\subset TM\oplus T^*M$
is a $\phi$-twisted Dirac structure \cite{courant,SW}, then the
projection $L\rmap T^*M$ is naturally an IM 2-form, so the
correspondence above includes the integration of twisted Dirac
structures as a special case.

The IM 2-form associated with a multiplicative 2-form $\omega\in
\Omega^2(\Grd)$ is simply
\begin{equation}\label{eq:sig}
\sigma(u)= i_{u}\omega|_{TM}, \;\;\; u\in A,
\end{equation}
where $A$ and $TM$ are naturally viewed as subbundles of $T\Grd|_M$.
The construction of $\omega$ from a given $\sigma:A\rmap T^*M$ in
\cite[Sec.~5]{bcwz} relies on the identification of $\Grd$ with
$A$-homotopy classes of $A$-paths (in the sense of \cite{CF}, cf.
\cite{severa}), in such a way that $\omega$ is obtained by a
variation of the infinite dimensional reduction procedure of
\cite{catfel}. A different, more general, viewpoint to this problem
has been recently studied in \cite{AC}, where this correspondence is
seen as part of a general Van Est isomorphism.

In this paper, we avoid the use of path spaces by noticing that the
construction of a multiplicative $\omega\in \Omega^2(\Grd)$ out of
an IM 2-form $\sigma$ can be phrased as the integration of a
suitable Lie algebroid morphism, similar in spirit to the approach
of Mackenzie and Xu \cite{Mac-Xu,Mac-Xu2} to the problem of
integrating Lie bialgebroids to Poisson groupoids, which served as
our main source of inspiration.

We notice that any multiplicative 2-form $\omega \in \Omega^2(\Grd)$
naturally induces a 2-form $\Lambda\in \Omega^2(A)$ on the total
space of $A$, which is \textit{linear} in a suitable sense. We show
that, when $\omega$ is relatively $\phi$-closed, the 2-form
$\Lambda$ is totally determined by the map $\sigma$ \eqref{eq:sig}
and $\phi$ via the formula
\begin{equation}\label{eq:Lambda}
\Lambda = - (\sigma^*\omega_{can} + \rho^*\tau(\phi)),
\end{equation}
where $\omega_{can}$ is the canonical symplectic form on $T^*M$, and
$\tau(\phi)\in \Omega^2(TM)$
is the 2-form defined, at each point
$X\in TM$, by $\tau(\phi)|_X=p_M^*(i_X\phi)$, where $p_M:TM\rmap M$
denotes the natural projection.

As a key step to reconstructing multiplicative 2-forms from
infinitesimal data, consider an arbitrary Lie algebroid $A\rmap M$,
together with a  vector-bundle map $\sigma: A\rmap T^*M$ and a
closed form $\phi \in \Omega^3(M)$. Let us use $\sigma$ and $\phi$
to define $\Lambda\in \Omega^2(A)$ by \eqref{eq:Lambda}. Our main
observation is that the bundle map
$$
\Lambda^\sharp:TA\rmap T^*A, \;\; U\mapsto i_U\Lambda,
$$
is a morphism between tangent and cotangent Lie algebroids (see
\cite{Mac-Xu}) if and only if $\sigma$ is an IM 2-form relative to
$\phi$. This result can be immediately applied to the integration of
IM 2-forms: the morphism of groupoids $T\Grd\rmap T^*\Grd$ obtained
by integrating the morphism $\Lambda^\sharp:TA\rmap T^*A$ determines
the desired multiplicative 2-form. Our approach to multiplicative
2-forms can be naturally extended in different directions, e.g., to
forms of higher degree or forms with no
prescription on their exterior derivatives, as
recently done in \cite{AC} from a different perspective. These extensions
and a comparison with \cite{AC} will be discussed in a
separate paper.

The paper is organized as follows. In Section \ref{sec:tancot} we
briefly recall the definitions and main properties of tangent and
cotangent Lie algebroids and groupoids. In Section
\ref{sec:liftlie}, we discuss the construction of linear 2-forms on
Lie algebroids associated with multiplicative 2-forms on Lie
groupoids. In Section \ref{sec:mult}, we relate IM 2-forms with
linear 2-forms defining algebroid morphisms $TA\rmap T^*A$, and
apply our results to the integration of IM 2-forms.

\subsection{Notations and conventions}
For a Lie groupoid $\Grd$ over $M$, its source and target maps are
denoted by $\sour$, $\tar$. Composable pairs $(g,h)\in
\Grd^{(2)}=\Grd\times_M \Grd$ are such that $\sour(g)=\tar(h)$, and
the multiplication map is denoted by $m:\Grd^{(2)}\to \Grd$,
$m(g,h)=gh$. Its Lie algebroid is $A\Grd=\Ker(Ts)|_M$, with anchor
$T\tar|_A:A\rmap TM$, and bracket induced by right-invariant vector
fields. For a general Lie algebroid $A\rmap M$, we denote its anchor
by $\rho_A$ and bracket by $[\cdot,\cdot]_A$ (or simply $\rho$ and
$[\cdot,\cdot]$ if there is no risk of confusion). Given vector
bundles $A\to M$ and $B\to M$, \textit{vector-bundle maps} $A\to B$
in this paper are assumed to cover the identity map, unless
otherwise stated. Einstein's summation convention is consistently
used throughout the paper.

\subsection{Acknowledgements}
Bursztyn and Cabrera thank CNPq (the Brazilian National Research
Council) for financial support. Ortiz was supported by a PEC-PG
scholarship from CAPES. We thank Yvette Kosmann-Schwarzbach and the
referees for their valuable comments.

\section{Tangent and cotangent structures}\label{sec:tancot}

In this section, we briefly recall tangent and cotangent algebroids
and groupoids, following \cite{Mac-book,Mac-Xu}, where readers can
find more details.

\subsection{Tangent and cotangent Lie groupoids}

Let $\Grd$ be a Lie groupoid over $M$, with Lie algebroid $A\Grd$
(if there is no risk of confusion, we may denote $A\Grd$ simply by
$A$). The tangent bundle $T\Grd$ has a natural Lie groupoid
structure over $TM$, with source (resp., target) map given by
$T\sour:T\Grd\rmap TM$ (resp., $T\tar:T\Grd\rmap TM$). The
multiplication on $T\Grd$ is defined by $Tm : T\Grd^{(2)} =
(T\Grd)^{(2)} \rmap T\Grd$. We refer to this groupoid as the
\textbf{tangent groupoid} of $\Grd$.

The cotangent bundle $T^*\Grd$ has a Lie groupoid structure over
$A^*$, known as the \textbf{cotangent groupoid} of $\Grd$. The
source and target maps are given by
$$
\tilde{\sour}(\alpha_g)u=\alpha_g(Tl_g(u-T\tar(u))),\qquad
\tilde{\tar}(\beta_g)v=\beta_g(Tr_g(v)),
$$
where $\alpha_g, \beta_g \in T^*_g\Grd$, $u\in A_{s(g)}$, and $v\in
A_{t(g)}$. Here $l_g: \tar^{-1}(\sour(g))\rmap \tar^{-1}(\tar(g))$
and $r_g: \sour^{-1}(\tar(g))\rmap \sour^{-1}(\sour(g))$ denote the
left and right multiplications by $g\in \Grd$, respectively. The
multiplication on $T^*\Grd$, denoted by $\circ$, is defined by
\begin{equation}\label{eq:cotgmult}
\alpha_g\circ \beta_h(Tm(X_g,Y_h))= \alpha_g(X_g)+ \beta_h(Y_h),
\end{equation}
for $(X_g,Y_h)\in T_{(g,h)}\Grd^{(2)}$.

\subsection{Tangent double vector bundles and duals}
\label{subsec:DVB}

Let $q_A: A \rmap M$ be a vector bundle. There is a natural
\textit{double vector bundle} \cite{Mac-book,Prad} associated with
it, referred to as the \textbf{tangent double vector bundle} of $A$,
and defined by the following diagram:
\begin{equation}\label{eq:doubletangent}
\xymatrix{ TA  \ar[r]^{T{q_A}} \ar[d]_{p_A} & TM \ar[d]^{p_M} \\
A \ar[r]_{q_A} &  M .}
\end{equation}
Here the vertical arrows are the usual tangent bundle structures.
Similarly, one can consider the tangent double vector bundle of
$q_{A^*}: A^* {\rmap} M$, which defines a double vector bundle
$TA^*$:
\begin{equation}\label{eq:doublecotangent}
\xymatrix{ TA^*  \ar[r]^{T{q_{A^*}}} \ar[d]_{p_{A^*}} & TM \ar[d]^{p_M} \\
A^* \ar[r]_{q_{A^*}} &  M .}
\end{equation}

It will be useful to consider coordinates on these bundles. If
$(x^j)$, $j=1,\ldots,\mbox{dim}(M)$, are local coordinates on $M$
and $\{e_d\}$, $d=1,\ldots,\mbox{rank}(A)$, is a basis of local
sections of $A$, we write the corresponding coordinates on $A$ as
$(x^j,u^d)$ and tangent coordinates on $TA$ as
$(x^j,u^d,\dot{x}^j,\dot{u}^d)$. For each $x=(x^j)$, note that
$(u^d)$ specifies a point in $A_x$, $(\dot{x}^j)$ gives a point in
$T_xM$, whereas $(\dot{u}^d)$ determines a point on a second copy of
$A_x$, tangent to the fibres of $A\rmap M$, known as the
\textit{core} of $TA$ (defined by $\mbox{ker}(p_A)\cap
\mbox{ker}(Tq_A)$, see \cite{Mac-book,Prad}). Similarly, we have
local coordinates $(x^j,\xi_d)$ on $A^*$ (relative to the basis
$\{e^d\}$, dual to $\{e_d\}$), and tangent coordinates
$(x^j,\xi_d,\dot{x}^j,\dot{\xi}_d)$, where now the coordinates
$(\dot{\xi}_d)$ represent the core directions.

Let $T^\bullet A\rmap TM$ be the vector bundle defined by dualizing
the fibres of $Tq_A: TA \rmap TM$,
$(x^j,u^d,\dot{x}^j,\dot{u}^d)\mapsto (x^j,\dot{x}^j)$. This fits
into the double vector bundle
\begin{equation}\label{eq:TMdualtangentprolongation}
\xymatrix{ T^\bullet A  \ar[r] \ar[d] & TM \ar[d]^{p_M} \\
A^* \ar[r]_{q_{A^*}} &  M .}
\end{equation}
Here the vertical map $T^\bullet A\to A^*$ is defined by
$(x^j,\zeta_d,\dot{x}^j,\eta_d)\mapsto (x^j,\eta_d)$, where
$T^\bullet A$ is locally written as
$(x^j,\zeta_d,\dot{x}^j,\eta_d)$, with $(\zeta_d)$ dual to $(u^d)$,
and $(\eta_d)$ dual to $(\dot{u}^d)$.

The double vector bundles \eqref{eq:doublecotangent} and
\eqref{eq:TMdualtangentprolongation} turn out to be isomorphic: as
shown in \cite[Proposition~5.3]{Mac-Xu}, by applying the tangent
functor to the natural pairing $A^*\times_{M} A\rmap \reals$
(followed by the fibre projection $T\mathbb{R}\to \mathbb{R}$) one
obtains a nondegenerate pairing $TA^*\times_{TM} TA\rmap  \reals$,
which induces an isomorphism of double vector bundles
\begin{equation}\label{eq:I}
I: TA^* \rmap T^\bullet A .
\end{equation}
Locally, this identification amounts to the flip
$$
(x^j,\xi_d,\dot{x}^j,\dot{\xi}_d)\mapsto
(x^j,\dot{\xi}_d,\dot{x}^j,{\xi}_d).
$$

The cotangent bundle $T^*A$ can be locally written in coordinates
$(x^j,u^d,p_j,\zeta_d)$, where $(p_j)$ determines a point in
$T^*_xM$ and $\zeta_d$ in $A^*_x$ (dual to the direction tangent to
the fibres $A\rmap M$). If $c_A: T^*A\rmap A$,
$c_A(x^j,u^d,p_j,\zeta_d)=(x^j,u^d)$ denotes the natural projection,
we see that $T^*A$ fits into the following double vector bundle:
\begin{equation}\label{eq:dualtangentprolongation}
\xymatrix{ T^* A  \ar[r]^r \ar[d]_{c_A} & A^* \ar[d]^{q_{A^*}} \\
A \ar[r]_{q_{A}} &  M ,}
\end{equation}
where the bundle projection $r:T^*A\rmap A^*$ is given locally by
$r(x^j,u^d,p_j,\zeta_d)=(x^j,\zeta_d)$. The same construction can be
applied to the vector bundle $A^*\rmap M$, yielding a double vector
bundle structure for $T^* A^*$. These double vector bundles can be
identified by a Legendre type transform \cite[Thm.~5.5]{Mac-Xu} (cf.
\cite{tulc}):
\begin{equation}\label{eq:R}
R: T^*A^*\rmap T^*A,
\end{equation}
given locally by $(x^j,\xi_d,p_j,u^d)\mapsto (x^j,u^d,-p_j,\xi_d)$.

There are two other identifications involving tangent and cotangent
double vector bundles that we need to recall. For an arbitrary
manifold $M$, we first have the \textit{canonical involution}
\begin{equation}\label{eq:canonicalinvolution}
\xymatrix{ T T M  \ar[r]^{J_M} \ar[d]_{p_{TM}} & TTM \ar[d]^{Tp_M} \\
TM \ar[r]_{\Id} &  TM ,}
\end{equation}
which is an isomorphism of double vector bundles (restricting to the
identity on side bundles and cores). Writing local coordinates
$(x^j,\dot{x}^j)$ for $TM$, and tangent coordinates
$(x^j,\dot{x}^j,\delta x^j, \delta \dot{x}^j)$ for $T(TM)$, $J_M$ is
given by
$$
J_M(x^j,\dot{x}^j,\delta x^j, \delta \dot{x}^j) = (x^j,\delta x^j,
\dot{x}^j, \delta \dot{x}^j).
$$
There is also an isomorphism of double vector bundles (also
restricting to the identity on side bundles and cores),
\begin{equation}\label{eq:flipTT*}
\Theta_M: T T^* M \rmap T^* T M,
\end{equation}
defined in local coordinates by
$$
\Theta_M(x^j,p_j,\dot{x}^j, \dot{p}_j)=
(x^j,\dot{x}^j,\dot{p}_j,p_j).
$$
Here $(x^j,p_j)$ are cotangent coordinates on $T^*M$. Equivalently,
$\Theta_M = J_M^*\circ I_M$, where $J_M^*: T^\bullet TM \rmap T^*TM$
is the dual of \eqref{eq:canonicalinvolution}, and
\begin{equation}\label{eq:I_M}
I_M: TT^*M \rmap T^\bullet TM
\end{equation}
is as in \eqref{eq:I} (with $A=TM$).

\subsection{Tangent and cotangent Lie
algebroids}\label{subsec:tancot}

Suppose that the vector bundle $A\rmap M$ carries a Lie algebroid
structure, which can be equivalently described by a fibrewise linear
Poisson structure on $A^*$ (see, e.g., \cite[Sec.~16.5]{CW}). Since
any Poisson structure on a manifold defines a Lie algebroid
structure on its cotangent bundle (see, e.g., \cite[Sec.~17.3]{CW}),
we obtain a Lie algebroid structure on $T^*A^*$; it follows that
$TA^*$ inherits a Poisson structure, which turns out to be linear
with respect to \textit{both} vector bundle structures on $TA^*$
\eqref{eq:doublecotangent}. Hence the vector bundle $T^\bullet A^*
\rmap TM$, dual to $TA^* \rmap TM$, is a Lie algebroid. Using the
identification $T^\bullet A^* \cong TA$ as in \eqref{eq:I}, we
obtain a Lie algebroid structure on $TA\rmap TM$, referred to as the
\textbf{tangent Lie algebroid} of $A$.

To describe this algebroid structure more explicitly, we recall that
any section $u\in \Gamma(A)$ gives rise to two types of sections on
$TA$: the first one is just $Tu: TM \to TA$, and the second one,
denoted by $\widehat{u}$, identifies $u$ at each point with a core
element in $TA$; locally, using coordinates $(x^j,u^d)$ for $A$ and
$(x^j,u^d,\dot{x}^j,\dot{u}^d)$ for $TA$, $\widehat{u}:TM\to TA$ is
defined by
\begin{equation}\label{eq:coreT}
\widehat{u}(x^j,\dot{x}^j) = (x^j, 0, \dot{x}^j, u^d(x)).
\end{equation}
These two types of sections generate the space of sections of
$TA\rmap TM$. The Lie algebroid structure on $TA$ is completely
described in terms of these sections by the relations \cite{Mac-Xu}:
\begin{equation}\label{eq:TArel}
[\widehat{u},\widehat{v}]_{\st{TA}}=0, \;\;
[Tu,\widehat{v}]_{\st{TA}}=\widehat{[u,v]}_{\st{A}},\;\;
[Tu,Tv]_{\st{TA}}=T [u,v]_{\st{A}},
\end{equation}
for $u,v \in \Gamma(A)$; the anchor map is $\rho_{TA}= J_M\circ
T\rho_A$, where $J_M:T(TM)\rmap T(TM)$ is as in
\eqref{eq:canonicalinvolution}.

On the other hand, since $T^* A^* \to A^*$ is a Lie algebroid
(defined by the linear Poisson structure on $A^*$), one can induce a
Lie algebroid structure on $r: T^*A \to A^*$ using the
identification \eqref{eq:R}. This is known as the \textbf{cotangent
Lie algebroid} of $A$. Explicit formulas for its bracket and anchor
will be recalled in Section \ref{subsec:IMmorphism}.

Suppose that $A=A\Grd$ is the Lie algebroid of a Lie groupoid
$\Grd$, and consider the natural inclusion $\iota_{A\Grd}:A\Grd\rmap
T\Grd$, which is a bundle map over the unit map $M\hookrightarrow
\Grd$. Then the canonical involution $J_\Grd:T(T\Grd)\rmap T(T\Grd)$
\eqref{eq:canonicalinvolution} restricts to a Lie algebroid
isomorphism
\begin{equation}\label{eq:j}
j_{\Grd}:T(A\Grd)\rmap A(T\Grd).
\end{equation}
In other words, we have a commutative diagram
\begin{equation}
\xymatrix{ T (A\Grd)  \ar[r]^{j_{\Grd}} \ar[d]_{T\iota_{A\Grd}} & A(T\Grd) \ar[d]^{\iota_{A(T\Grd)}} \\
T(T\Grd) \ar[r]_{J_\Grd} &  T(T\Grd) .}
\end{equation}
The canonical pairing $T^*\Grd\times_{\Grd}T\Grd\rmap \reals$ is a
morphism of groupoids, and applying the Lie functor one obtains a
nondegenerate pairing $A(T^*\Grd)\times_{A\Grd}A(T\Grd)\rmap
\reals$, explicitly given by
$$
\SP{U,V} = \SP{I_\Grd(\iota_{A(T^*\Grd)}(U)),\iota_{A(T\Grd)}(V)},
$$
where $U \in A(T^*\Grd)$, $V \in A(T\Grd)$, and $I_\Grd$ is as in
\eqref{eq:I_M}. This induces an isomorphism $A(T^*\Grd)\rmap
A^\bullet(T\Grd)$, where $A^\bullet(T\Grd)$ is obtained by dualizing
the fibres of $A(T\Grd)\rmap A(\Grd)$, and the composition of this
map with $j^*_{\Grd}: A^\bullet(T\Grd) \rmap T^*(A\Grd)$ defines a
Lie algebroid isomorphism
\begin{equation}\label{eq:thetag}
\theta_{\Grd}:A(T^*\Grd)\rmap T^*(A\Grd).
\end{equation}
Alternatively, one can check that $\theta_\Grd =
(T\iota_{A\Grd})^t\circ \Theta_{\Grd}\circ \iota_{A(T^*\Grd)}$,
where $(T\iota_{A\Grd})^t:\iota_{A\Grd}^*T^*(T\Grd)\rmap T^*(A\Grd)$
is dual to the tangent map $T\iota_{A\Grd}: T(A\Grd)\rmap
\iota_{A\Grd}^*T(T\Grd)$.

\section{Tangent lifts and the Lie functor}\label{sec:liftlie}

We now discuss how multiplicative forms on Lie groupoids relate to
differential forms on Lie algebroids. As a first step, we need to
recall a natural operation that lifts differential forms on a
manifold to its tangent bundle,
\begin{equation}
\Omega^k(M) \rmap \Omega^k(TM), \;\; \alpha \mapsto \alpha_T,
\end{equation}
known as the \textbf{tangent (or complete) lift}, see \cite{GU,YI}.

\subsection{Tangent lifts of differential forms}
The properties of tangent lifts recalled in this subsection can be
found (often in more generality) in \cite{GU}; we included the
proofs of some key facts for the sake of completeness.

Given the tangent bundle $p_M:TM\rmap M$, $(x^j,\dot{x}^j)\mapsto
(x^j)$, consider the two vector bundle structures associated with
$T(TM)$:
\begin{equation}
\xymatrix{ T (TM)  \ar[r]^{Tp_M} \ar[d]_{p_{TM}} & TM \\
TM,  & }
\end{equation}
where $p_{TM}(x^j,\dot{x}^j,\delta x^j, \delta
\dot{x}^j)=(x^j,\dot{x}^j)$ and $Tp_M(x^j,\dot{x}^j,\delta x^j,
\delta \dot{x}^j)=(x^j,\delta x^j)$. We use the notation
$$
T(TM)\times_{\st{Tp_M}} T(TM), \;\;\; T(TM)\times_{\st{p_{TM}}}
T(TM),
$$
to specify the vector bundle structure used for fibre products over
$TM$; more general $k$-fold fibre products over $TM$ are denoted by
$$
{\prod_{T{p_M}}^k}T(TM), \;\;\; {\prod_{p_{TM}}^k}T(TM).
$$
Using the involution \eqref{eq:canonicalinvolution}, given by
$J_M(x^j,\dot{x}^j,\delta x^j, \delta \dot{x}^j) = (x^j,\delta
x^j,\dot{x}^j, \delta \dot{x}^j)$ in local coordinates, we obtain a
natural isomorphism
\begin{equation}
J_M^{(k)}: {\prod_{p_{TM}}^k} T(TM) \rmap {\prod_{Tp_M}^k}T(TM).
\end{equation}

Given a $k$-form $\alpha \in \Omega^k(M)$, $k\geq 1$, consider the
bundle map
\begin{equation}\label{eq:sharp}
\alpha^\sharp: \prod_{p_M}^{k-1}TM\rmap T^*M, \;\;
\alpha^\sharp(X_1,\ldots,X_{k-1}) = i_{X_{k-1}}\ldots i_{X_1}\alpha.
\end{equation}
(For $k=1$, $\alpha^\sharp: M \rmap T^*M$ is just $\alpha$ viewed as
a section of $T^*M$.) Using the natural identification $T
(\prod^k_{p_M} TM)=\prod^k_{Tp_M} T(TM)$,
we consider the tangent map
$$
T\alpha^\sharp: \prod_{Tp_M}^{k-1}T(TM) \rmap T(T^*M).
$$
The \textbf{tangent (or complete) lift} of a $k$-form on $M$ is
defined as follows (cf. \cite{YI}):
\begin{itemize}
\item If $f\in \Omega^0(M)=C^\infty(M)$, then $f_T\in C^\infty(TM)$ is the fibrewise linear function
on $TM$ defined by $\mathrm{d}f$,
$$
f_T(X)=(\mathrm{d}f)_{p_M(X)}(X), \;\; X\in TM.
$$

\item If $\alpha \in \Omega^k(M)$, $k\geq 1$,
we define
$$
(\alpha_T)^\sharp: \prod^{k-1}_{p_{TM}} T(TM) \rmap T^*(TM), \;\;\;
(\alpha_T)^\sharp:= \Theta_M \circ T\alpha^\sharp \circ J_M^{(k-1)},
$$
and then $\alpha_T\in \Omega^k(TM)$ is given by
$$
\alpha_T(U_1,\ldots,U_k):=\SP{\alpha_T^\sharp(U_1,\ldots,U_{k-1}),U_k}.
$$
\end{itemize}

One can directly verify that $\alpha_T$ is multilinear. The fact
that it is indeed a $k$-form on $TM$ follows from the next lemma
(cf. \cite{GU,YI}).

\begin{lemma}\label{lem:properties}
The following holds:
\begin{enumerate}
\item[$(i)$] For $f\in C^\infty(M)$, $\mathrm{d} f_T = (\mathrm{d}f)_T$.

\item[$(ii)$] For $f\in C^\infty(M)$, $\alpha \in \Omega^k(M)$,
\begin{equation*}
(f\alpha)_T= f_T \alpha^\vee + f^\vee \alpha_T,
\end{equation*}
where $\beta^\vee = p_M^*\beta$ for any $\beta\in \Omega^l(M)$.

\item[$(iii)$] For $k\geq 2$, the tangent lift $(\mathrm{d}x^{i_1}\wedge \ldots \wedge
\mathrm{d}x^{i_k})_T$ equals
$$
\sum_{m=1}^k (\mathrm{d}x^{i_1})^\vee\wedge \ldots \wedge
(\mathrm{d}x^{i_{m-1}})^\vee\wedge (\mathrm{d}x^{i_m})_T\wedge
(\mathrm{d}x^{i_{m+1}})^\vee\wedge \ldots\wedge (\mathrm{d}x^{i_k})^\vee.
$$
(Whenever there is no risk of confusion, we write $(\mathrm{d}x^j)^\vee$
simply as $\mathrm{d}x^j$.)

\end{enumerate}
\end{lemma}

\begin{proof}
To verify $(i)$, let us consider $X\in TM$ and $U \in T_X(TM)$. In
local coordinates, we write $X=(x^j,\dot{x}^j)$ and
$U=(x^j,\dot{x}^j,\delta x^j, \delta \dot{x}^j)$. Then
$f_T(X)=\frac{\partial f}{\partial x^i} \dot{x}^i$, and
\begin{equation}\label{eq:df}
\mathrm{d}(f_T)_X(U)= \frac{\partial^2f}{\partial x^j \partial x^i}\dot{x}^i
\delta x^j + \frac{\partial f}{\partial x^j}\delta \dot{x}^j.
\end{equation}
On the other hand, we may view $\mathrm{d}f$ as a section
$$
(\mathrm{d}f)^\sharp: M\rmap T^*M, \;\;\; x=(x^j) \mapsto (x^j,
\frac{\partial f}{\partial x^j}).
$$
Hence $T(\mathrm{d}f)^\sharp: TM \rmap T(T^*M)$ is given by
$$
T(\mathrm{d}f)^\sharp(x^j,\dot{x}^j)=(x^j,\frac{\partial f}{\partial
x^j},\dot{x}^j,\frac{\partial^2 f}{\partial x^i\partial x^j}
\dot{x}^i),
$$
and, as a consequence,
$$
(\mathrm{d}f)_T(x^j,\dot{x}^j)=\Theta_M(T(\mathrm{d}f)^\sharp(x^j,\dot{x}^j))=(x^j,\dot{x}^j,\frac{\partial^2
f}{\partial x^i\partial x^j} \dot{x}^i,\frac{\partial f}{\partial
x^j}).
$$
It immediately follows that $((\mathrm{d}f)_T)_X(U)$ agrees with
\eqref{eq:df}.

Let us show that $(ii)$ holds for $k> 1$ (the cases $k=0,1$ are
simpler). One can directly check that $(f \alpha)^\sharp = f
\alpha^\sharp$ and
$$
T(f\alpha)^\sharp(U_1,\ldots,U_{k-1})=
\alpha^\sharp(X_1,\ldots,X_{k-1}) (\mathrm{d}f)_x(Y) + f(x)
T\alpha^\sharp(U_1,\ldots,U_{k-1}),
$$
where $X_i = p_{TM}(U_i) \in T_xM$, and $Y= (p_M)_*(U_1)=\ldots=
(p_M)_*(U_{k-1})$. In the last formula, addition and multiplication
by scalars are with respect to the vector bundle structure
$T(T^*M)\rmap TM$ (in the fibre over $Y\in T_xM$), and
$\alpha^\sharp(X_1,\ldots,X_{k-1}) \in T^*_xM$ is viewed inside
$T(T^*M)$ as the core (i.e., tangent to $T^*M$-fibres). Since
$\Theta_M: T(T^*M)\rmap T^*(TM)$ is a double vector bundle
isomorphism restricting to the identity on side bundles and cores,
we have
\begin{align}
\Theta_M T(f\alpha)^\sharp(U_1,\ldots,U_{k-1})= \label{eq:leib}&
\alpha^\sharp(X_1,\ldots,X_{k-1}) (\mathrm{d}f)_x(Y)\\
& + f(x) \Theta_M T\alpha^\sharp(U_1,\ldots,U_{k-1}), \nonumber
\end{align}
where now the addition and scalar multiplication operations are
relative to the vector bundle $T^*(TM)\rmap TM$, and
$\alpha^\sharp(X_1,\ldots,X_{k-1})$ belongs to the core fibre in
$T^*(TM)$ (i.e., cotangent to $M$). Writing $(U_1,\ldots,U_{k-1}) =
J_M^{(k-1)}(V_1,\dots,V_{k-1})$, then $X_i=(p_M)_* (V_i)$ and
$Y=p_{TM} (V_i)$, so \eqref{eq:leib} yields
$$
(f\alpha)_T^\sharp = (f_T \alpha^\vee + f^\vee \alpha_T)^\sharp.
$$

Let us now prove $(iii)$. Note that
$$
(\mathrm{d}x^{i_1}\wedge \ldots \wedge \mathrm{d}x^{i_k})^\sharp(X_1,\dots,X_{k-1})=
\sum_{\sigma \in S_k} (-1)^\sigma \dot{x}_1^{i_{\sigma(1)}}\ldots
\dot{x}_{k-1}^{i_{\sigma(k-1)}}\mathrm{d}x^{i_{\sigma(k)}},
$$
where $X_l=(x^j,\dot{x}_l^j)\in T_xM$. Then
$\SP{\Theta_M(T(\mathrm{d}x^{i_1}\wedge\ldots \wedge
\mathrm{d}x^{i_k})^\sharp(U_1,\ldots,U_{k-1})),V_k}$ equals
\begin{equation}\label{eq:expression}
\sum_{\sigma \in S_k} (-1)^\sigma
\sum_{n=1}^k\dot{x}_1^{i_{\sigma(1)}}\ldots
\dot{x}_{n-1}^{i_{\sigma(n-1)}} (\delta \dot{x})_n^{i_{\sigma(n)}}
\dot{x}_{n+1}^{i_{\sigma(n+1)}} \ldots \dot{x}_k^{i_{\sigma(k)}},
\end{equation}
where $U_l=(x^j,\dot{x}_l^j,(\delta x)^j,(\delta \dot{x})_l^j)$ and
$V_k=(x^j,(\delta x)^j,\dot{x}_k^j,(\delta \dot{x})_k^j)$. Since
$(\mathrm{d}x^j)_T=\mathrm{d}\dot{x}^j$ (by $(i)$), one checks that
$$
\sum_{n=1}^k (\mathrm{d}x^{i_1})^\vee\wedge \ldots \wedge
(\mathrm{d}x^{i_{n-1}})^\vee\wedge (\mathrm{d}x^{i_n})_T\wedge
(\mathrm{d}x^{i_{n+1}})^\vee\wedge \ldots\wedge
(\mathrm{d}x^{i_k})^\vee(V_1,\ldots,V_k),
$$
where $V_l=(x^j,(\delta x)^j,\dot{x}_l^j,(\delta \dot{x})_l^j)$ (so
that $J_M(V_l)=U_l$), equals
$$
\sum_{\sigma \in S_k}(-1)^\sigma \sum_{n=1}^k
\dot{x}^{i_1}_{\sigma(1)}\ldots
\dot{x}^{i_{n-1}}_{\sigma(n-1)}(\delta \dot{x})^{i_n}_{\sigma(n)}
\dot{x}^{i_{n+1}}_{\sigma(n+1)}\ldots \dot{x}^{i_k}_{\sigma(k)},
$$
which agrees with \eqref{eq:expression} after reshuffling indices.


\end{proof}

Let us now consider the operation
\begin{equation}\label{eq:tau}
\tau: \Omega^k(M)\rmap \Omega^{k-1}(TM),\;\; \tau(\alpha)_X =
p_M^*(i_X\alpha),
\end{equation}
where $X \in TM$ and $k\geq 1$.
In other words, given $U_1,\ldots,U_{k-1}\in T_X(TM)$,
$$
\tau(\alpha)_X(U_1,\ldots,U_{k-1})=\alpha(X,(p_M)_*(U_1),\ldots,(p_M)_*(U_{k-1})).
$$
In coordinates, writing $\alpha = \frac{1}{k!}\alpha_{i_1\ldots
i_k}(x) \mathrm{d}x^{i_1}\wedge \ldots \wedge \mathrm{d}x^{i_k}$ (with
$\alpha_{i_1\ldots i_k}$ totally anti-symmetric), we have
$$
\tau(\alpha)_X = \frac{1}{(k-1)!}\alpha_{i_1\ldots i_k}(x) X^{i_1}
\mathrm{d}x^{i_2}\wedge \ldots \wedge \mathrm{d}x^{i_k}.
$$

\begin{example}\label{ex:can1form}
Consider the map $\omega^\sharp:TM\rmap T^*M$,
$\omega^\sharp(X)=i_X\omega$, associated with a 2-form $\omega \in
\Omega^2(M)$. A direct computation shows that
$$
\tau(\omega)=(\omega^\sharp)^* \theta_{can},
$$
where $\theta_{can}\in \Omega^1(T^*M)$ is the canonical 1-form,
$\theta_{can}=p_i\mathrm{d}x^i$.
\end{example}

The tangent lift can be computed by the following Cartan-like
formula (cf. \cite{GU}).

\begin{proposition}\label{prop:magic}
For $\alpha\in \Omega^k(M)$,  the tangent lift of $\alpha$ is given
by the formula
\begin{equation}\label{eq:magic}
\alpha_T = \mathrm{d}\tau(\alpha) + \tau(\mathrm{d}\alpha).
\end{equation}
\end{proposition}
\begin{proof}
It suffices to check \eqref{eq:magic} locally, so we replace $M$ by
a neighborhood with coordinates $(x^j)$, so that $TM$ has
coordinates $(x^j,\dot{x}^j)$. Let us consider the vector field $V$
on $TM$ defined by
$$
V_X := \dot{x}^j\frac{\partial}{\partial x^j} \; \in T_X(TM),
$$
where $X=(x^j,\dot{x}^j) \in TM$. This vector field has the property
that $Tp_M(V_X)=X$. One can directly check that
\begin{equation}\label{eq:V}
f_T = \Lie_V (p_M^* f),\;\;\mbox{ and } \;\;  (\mathrm{d}x^j)_T = \mathrm{d}\dot{x}^j
= \Lie_V (p_M^* \mathrm{d}x^j),
\end{equation}
where $f\in C^\infty(M)$. From the definition of $\tau$, it
immediately follows that
\begin{equation}\label{eq:tauV}
\tau(\beta)= i_V p_M^*\beta, \;\;\; \beta\in \Omega^k(M).
\end{equation}
Given $\alpha = \frac{1}{k!}\alpha_{i_1\ldots i_k}(x) \mathrm{d}x^{i_1}\wedge
\ldots \wedge \mathrm{d}x^{i_k}$, using Lemma \ref{lem:properties} we obtain
\begin{align*}
\alpha_T = & \frac{1}{k!} (\alpha_{i_1\ldots i_k})_T\,
p_M^*(\mathrm{d}x^{i_1}\wedge \ldots \wedge \mathrm{d}x^{i_k}) +
\frac{1}{k!}p_M^*\alpha_{i_1\ldots i_k} (\mathrm{d}x^{i_1}\wedge \ldots
\wedge
\mathrm{d}x^{i_k})_T\\
 = & \frac{1}{k!} (\alpha_{i_1\ldots i_k})_T \, p_M^*(\mathrm{d}x^{i_1}\wedge
\ldots \wedge \mathrm{d}x^{i_k}) +\\
& \frac{1}{k!}p_M^*\alpha_{i_1\ldots i_k} \sum_{n=1}^k
\mathrm{d}x^{i_1}\wedge \ldots \wedge (\mathrm{d}x^{i_n})_T \wedge \dots \wedge
\mathrm{d}x^{i_k}.
\end{align*}
It then follows from \eqref{eq:V} that $ \alpha_T = \Lie_V p_M^*
\alpha$. Using \eqref{eq:tauV} and Cartan's formula, we have
$$
\alpha_T = \mathrm{d} (i_V p_M^* \alpha) + i_V p_M^*\mathrm{d}\alpha = \mathrm{d}\tau(\alpha) +
\tau(\mathrm{d}\alpha).
$$
\end{proof}

\begin{example}\label{ex:closedT}
From Example \ref{ex:can1form}, it follows that if $\omega\in
\Omega^2(M)$, then
$$
\omega_T=-(\omega^\sharp)^*\omega_{can} + \tau(\mathrm{d}\omega).
$$
Here $\omega_{can}=-\mathrm{d}\theta_{can}=\mathrm{d}x^i\wedge \mathrm{d}p_i$ is the canonical
symplectic form on $T^*M$. (For the tangent lift of closed 2-forms,
see also \cite[Sec.~3]{courant2}).
\end{example}

An immediate consequence of \eqref{eq:magic} is the fact that
tangent lifts and exterior derivatives commute.

\begin{corollary}\label{cor:dcommute}
For $\alpha\in \Omega^k(M)$, $\mathrm{d}(\alpha_T)=(\mathrm{d}\alpha)_T$.
\end{corollary}


\subsection{Lie functor on multiplicative differential forms}

Let $\Grd$ be a Lie groupoid over $M$, $A=A\Grd$ its Lie algebroid,
and let $\alpha\in \Omega^k(\Grd)$. We can define an induced
$k$-form on $A$ by pulling back the tangent lift $\alpha_T \in
\Omega^k(T\Grd)$ via the inclusion $\iota_A:A\rmap T\Grd$. In this
section we discuss this operation when $\alpha$ is multiplicative.

Recall that a k-form $\alpha\in \Omega^k(\Grd)$ is
\textit{multiplicative} if
\begin{equation}\label{eq:mult}
m^*\alpha = p_1^*\alpha + p_2^*\alpha,
\end{equation}
where $p_1,p_2: \Grd^{(2)}\to \Grd$ are the natural projections, and
$m$ is the groupoid multiplication. We denote the associated
$k$-form on $A$ by
\begin{equation}\label{eq:Lieform}
\LF(\alpha) := \iota_A^*\alpha_T.
\end{equation}
Note that it follows from Corollary \ref{cor:dcommute} that
\begin{equation}
\mathrm{d} \LF(\alpha) = \LF(\mathrm{d} \alpha).
\end{equation}
In order to explain in which sense $\LF(\alpha)$ is the
infinitesimal counterpart of $\alpha$, we will need a known
alternative characterization of multiplicative forms.

The tangent groupoid structure on the tangent bundle $p_\Grd: T\Grd
\rmap \Grd$ over $TM$ induces a groupoid structure on the direct sum
$$
\prod_{p_{\Grd}}^n T\Grd = T\Grd \oplus \ldots \oplus T\Grd
$$
over the base $\prod_{p_M}^n TM=TM\oplus \ldots \oplus TM$ in a
canonical way.

\begin{lemma}\label{lem:morphism}
A $k$-form $\alpha \in \Omega^k(\Grd)$ ($k\geq 1$) is multiplicative
if and only if the bundle map $\alpha^\sharp:
\prod^{k-1}_{p_{\Grd}}T\Grd \rmap T^*\Grd$ (see \eqref{eq:sharp}) is
a groupoid morphism.
\end{lemma}

\begin{proof}
Let us consider the following identities, obtained by
differentiating basic identities on any Lie groupoid (see
\cite[Lem.~3.1]{bcwz}):
\begin{align}
&(Tm)_{(\tar(g),g)}(T\tar(X),X)= X = (Tm)_{(g,\sour(g))}(X,T\sour(X)),\;\; \forall X\in T_g\Grd, \label{eq:ident1}\\
&(Tr_g)_{\tar(g)}(u)=(Tm)_{(\tar(g),g)}(u,0),\;\;
(Tl_g)_{\sour(g)}(v)=(Tm)_{(g,\sour(g))}(0,v)\label{eq:ident2}
\end{align}
where $u \in A_{\tar(g)}=\mbox{Ker}(Ts)|_{\tar(g)}$ and $v \in
\mbox{Ker}(Tt)|_{\sour(g)}$. Using the first identities in
\eqref{eq:ident1} and \eqref{eq:ident2}, we see that if $\alpha$ is
multiplicative, then by \eqref{eq:mult} we have
$$
\alpha(T\tar(X_1),\ldots,T\tar(X_{k-1}),u)=
\alpha(X_1,\ldots,X_{k-1}, Tr_g(u)),
$$
where $X_i \in T_g\Grd, u\in A_{\tar(g)}.$ This is precisely the
compatibility of $\alpha^\sharp$ with the target maps on
$\prod^{k-1}_{p_{\Grd}}T\Grd$ and $T^*\Grd$. Similarly, note that
\eqref{eq:ident1} and \eqref{eq:mult} imply that, if $Z_1,\ldots,
Z_k \in TM$, then $\alpha(Z_1,\ldots,Z_k)=0$. Using this fact, along
with \eqref{eq:mult} and the second identities in \eqref{eq:ident1}
and \eqref{eq:ident2}, we obtain the compatibility of
$\alpha^\sharp$ and the source maps:
$$
\alpha(T\sour(X_1),\ldots,T\sour(X_{k-1}),u)=\alpha(X_1,\ldots,X_{k-1},Tl_g(u-T\tar(u))),
$$
where $X_i \in T_g\Grd, u\in A_{\tar(g)}$.

Assuming that $\alpha^\sharp$ is compatible with the source and
target maps, we see that it is a groupoid morphism if and only if
$$
\alpha^\sharp(Tm(X_1,Y_1),\ldots,Tm(X_{k-1},Y_{k-1}))=
\alpha^\sharp(X_1,\ldots,X_{k-1})\circ
\alpha^\sharp(Y_1,\ldots,Y_{k-1}).
$$
By evaluating each side of the last equation on $Tm(X_k,Y_k)$, we
see that this condition is equivalent to
$$
\alpha(Tm(X_1,Y_1),\ldots,Tm(X_k,Y_k))=
\alpha(X_1,\ldots,X_k)+\alpha(Y_1,\ldots,Y_k),
$$
which is precisely the multiplicativity condition \eqref{eq:mult}.
\end{proof}

Given a groupoid morphism $\psi: \Grd_1 \rmap \Grd_2$, we denote the
associated morphism of Lie algebroids (given by the restriction of
$T\psi: T\Grd_1 \rmap T\Grd_2$ to $A\Grd_1\subset T\Grd_1$) by
$$
\LF(\psi): A\Grd_1 \rmap A\Grd_2.
$$
The natural projection $p_\Grd : T\Grd\to \Grd$ is a groupoid
morphism, and one can directly verify that there is a canonical
identification
$$
A(\prod^{k-1}_{p_{\Grd}}T\Grd)=\prod^{k-1}_{\LF(p_\Grd)}A(T\Grd).
$$
Using this identification we get, for any given multiplicative
$k$-form $\alpha\in \Omega^k(\Grd)$, a Lie algebroid morphism
\begin{equation}
\LF(\alpha^\sharp): \prod^{k-1}_{\LF(p_\Grd)}A(T\Grd) \rmap
A(T^*\Grd).
\end{equation}
The isomorphism $j_\Grd: T (A\Grd){\rmap} A(T\Grd) $, see
\eqref{eq:j}, induces an identification
\begin{equation}
j_\Grd^{(k)}: \prod^k_{p_A} T(A\Grd) \rmap
\prod^k_{\LF(p_\Grd)}A(T\Grd).
\end{equation}
Recall the isomorphism $\theta_\Grd: A(T^*\Grd)\to T^* (A\Grd)$
defined in \eqref{eq:thetag}.

\begin{proposition}\label{prop:lie}
For a multiplicative $k$-form $\alpha\in \Omega^k(\Grd)$,
$\LF(\alpha)$ and $\LF(\alpha^\sharp)$ are related by
$$
\LF(\alpha)^\sharp = \theta_\Grd \circ \LF(\alpha^\sharp) \circ
j_\Grd^{(k-1)}: \prod^{k-1}_{p_A} T(A\Grd) \rmap T^*(A\Grd).
$$
\end{proposition}

\begin{proof}
Recall that $ \theta_{\Grd}=(T\iota_{A\Grd})^*\circ \Theta_\Grd
\circ \iota_{A(T^*\Grd)}$ and $J_{\Grd} \circ T\iota_{A\Grd}=
\iota_{A(T\Grd)}\circ j_{\Grd}$. This last identity immediately
implies that
$$
J^{(k)}_{\Grd} \circ (\prod^k T\iota_{A\Grd}) = (\prod^k
\iota_{A(T\Grd)})\circ j^{(k)}_{\Grd}.
$$
Since $\iota_{A(T^*\Grd)}\circ \LF(\alpha^\sharp)=  T\alpha^\sharp
\circ \prod^{k-1} \iota_{A(T\Grd)}$, it follows that
\begin{align*}
\theta_\Grd \circ \LF(\alpha^\sharp) \circ j_\Grd^{(k-1)} &=
(T\iota_{A\Grd})^t\circ \Theta_\Grd \circ T\alpha^\sharp \circ
\prod^{k-1} \iota_{A(T\Grd)} \circ j_\Grd^{(k-1)}\\
&= (T\iota_{A\Grd})^t \circ \alpha_T^\sharp \circ (\prod^{k-1}
T\iota_{A\Grd}),
\end{align*}
and this last term is $(\iota_A^*\alpha_T)^\sharp =
(\LF(\alpha))^\sharp$.
\end{proof}

\begin{corollary}\label{cor:lieunique}
If $\alpha \in \Omega^k(\Grd)$ is multiplicative and $\Grd$ is
$\sour$-connected, then $\alpha = 0$ if and only if $\LF(\alpha)=0$.
\end{corollary}

\begin{proof}
If $\Grd$ is $\sour$-connected, then $\prod^{k-1}_{p_\Grd}T\Grd$
also has connected source-fibres. We now use the fact that if two
groupoid morphisms $\Grd_1\rmap \Grd_2$ induce the same Lie
algebroid morphism and $\Grd_1$ has source-connected fibres, then
they must coincide. Hence $\alpha^\sharp=0$ if and only if
$\LF(\alpha^\sharp)=0$. The conclusion now follows since $\alpha=0$
(resp., $\LF(\alpha)=0$) is equivalent to $\alpha^\sharp =0$ (resp.,
$\LF(\alpha)^\sharp=0$), and $\LF(\alpha)^\sharp=0$ if only if
$\LF(\alpha^\sharp)=0$ by Proposition~\ref{prop:lie}.
\end{proof}

\section{Multiplicative 2-forms and their infinitesimal
counterparts}\label{sec:mult}

\subsection{Linear 2-forms on vector bundles}

Let $q:A\rmap M$ be a vector bundle, and consider the double vector
bundles $TA$ and $T^*A$, as in Section \ref{subsec:DVB}. A 2-form
$\Lambda \in \Omega^2(A)$ is called \textbf{linear} if
$$
\Lambda^\sharp:TA \rmap T^*A
$$
is a morphism of double vector bundles (cf. \cite[Sec.~7.3]{KU}). In
particular, there is a vector bundle map $\lambda: TM\to A^*$ (over
the identity)  such that the following diagram is commutative:
\begin{equation}\label{eq:diag}
\xymatrix{ TA  \ar[r]^{\Lambda^\sharp} \ar[d]_{Tq} & T^*A \ar[d]^{r} \\
TM \ar[r]_{\lambda} &  A^* .}
\end{equation}
In this case we say that $\Lambda$ \textbf{covers} $\lambda$.

\begin{remark}
The fact that a bivector field $\pi$ on a vector bundle $A$ is
linear  is equivalent \cite{KU, Mac-Xu} to the bundle map
$\pi^\sharp: T^*A\to TA$ being a morphism of double vector bundles.
Hence linear 2-forms are just their dual analogues.
\end{remark}

It is simple to check from the definition that a linear 2-form has a
local expression of the form:
\begin{align}
\Lambda & = \frac{1}{2}\Lambda_{ij}(x,u)\mathrm{d}x^i\wedge \mathrm{d}x^j + \Lambda_{jd}(x,u) \mathrm{d}x^j\wedge \mathrm{d}u^d \nonumber\\
&= \frac{1}{2}\Lambda_{ij,d}(x) u^d \mathrm{d}x^i\wedge \mathrm{d}x^j +
\lambda_{jd}(x) \mathrm{d}x^j\wedge \mathrm{d}u^d \label{eq:local}.
\end{align}
where $(x,u)=(x^j, u^d)$ are local coordinates in $A$ (relative to a
local basis $\{e_d\}$), and
$\lambda_{jd}=\SP{\lambda(\frac{\partial}{\partial {x^j}}),e_d}$.

\begin{example}\label{ex:linear}
The canonical symplectic form $\omega_{can}= \mathrm{d}x^j\wedge \mathrm{d}p_j$ on the
cotangent bundle $T^*M$ is linear. Any vector bundle map $\sigma: A
\rmap T^*M$, locally written as $\sigma(e_d)=\sigma_{jd}\mathrm{d}x^j$,
defines a linear 2-form on $A$ by pullback,
$$
\sigma^*\omega_{can} = u^d\frac{\partial \sigma_{id}}{\partial
x^k}\mathrm{d}x^i\wedge \mathrm{d}x^k + \sigma_{id}\mathrm{d}x^i\wedge \mathrm{d}u^d,
$$
covering the map $\lambda=\sigma^t: TM\rmap A^*$, where $\sigma^t$
is the fibrewise transpose of $\sigma$.
\end{example}

From the local expression \eqref{eq:local}, one can directly verify
that Example \ref{ex:linear} completely characterizes linear closed
2-forms:

\begin{proposition}\label{prop:linear}
A linear 2-form $\Lambda \in \Omega^2(A)$ is closed if and only if
it is of the form
$$
\Lambda = (\lambda^t)^* \omega_{can},
$$
where $\lambda^t: A \to T^*M$ is the fibrewise transpose of the
vector-bundle map $\lambda : TM \to A^*$ (see \eqref{eq:diag}).
\end{proposition}
A proof of this result can be found in \cite[Sec.~7.3]{KU}.

\begin{example}\label{ex:omegaT}
If $\omega\in \Omega^2(M)$, then its tangent lift $\omega_T\in
\Omega^2(TM)$ is linear and covers the map $\lambda=\omega^\sharp:
TM \to T^*M$. If $\omega$ is closed, then so is $\omega_T$ (it is in
fact exact, by Proposition \ref{prop:magic}). It follows from
Proposition \ref{prop:linear} and the fact that
$(\omega^\sharp)^t=-\omega^\sharp$ that
$$
\omega_T= -(\omega^\sharp)^* \omega_{can},
$$
in agreement with Example \ref{ex:closedT}.
\end{example}

\begin{example}\label{ex:twist}
Let $\phi \in \Omega^3(M)$ be a 3-form on $M$. Then the 2-form
$\tau(\phi)$ on $TM$  defined by \eqref{eq:tau} is linear; it covers
the bundle map $\lambda: A\to T^*M$ that is zero on each fibre.
\end{example}


\subsection{Linear 2-forms on Lie algebroids}

Let $A\rmap M$ be a Lie algebroid. We will discuss two natural ways
to obtain linear 2-forms on $A$.

First, given any 3-form $\phi \in \Omega^3(M)$, we can use the
anchor $\rho: A \rmap TM$ to pull-back the linear 2-form
$\tau(\phi)$ to $A$. The resulting 2-form
$$
\rho^*(\tau(\phi))\in \Omega^2(A)
$$
is linear, covering the map $\lambda: TM\rmap A^*$ that is zero on
each fibre.

On the other hand, if $A=A\Grd$ is the Lie algebroid of a Lie
groupoid $\Grd$, then one obtains linear 2-forms on $A$ as
infinitesimal versions of multiplicative 2-forms on $\Grd$:

\begin{proposition}\label{prop:mult}
Let $\omega \in \Omega^2(\Grd)$ be a multiplicative 2-form, and let
$\lambda:TM \rmap A^*$ be
defined by
$\lambda(X)(u)=\omega(X,u)$, for $X\in TM$ and $u \in A$. Then
\begin{enumerate}
\item $\Lambda=\LF(\omega) \in \Omega^2(A)$ is linear and covers $\lambda$.

\item Given $\phi\in
\Omega^3(M)$ closed and if $\Grd$ is $\sour$-connected, then
$\mathrm{d}\omega= \sour^*\phi - \tar^*\phi$ if and only if
$$
\Lambda= (\lambda^t)^* \omega_{can} - \rho^*(\tau(\phi)).
$$
\end{enumerate}
\end{proposition}

\begin{proof}
Let us prove $(1)$. Note that $\LF(\omega)=\iota_A^*\omega_T$ is
linear since $\omega_T \in \Omega^2(T\Grd)$ is linear, and the pull
back of a linear 2-form to a vector subbundle is again linear.

From Lemma~\ref{lem:morphism}, we know that $\omega^\sharp: T\Grd
\rmap T^*\Grd$ is a groupoid morphism, which restricts to the map
$\lambda:TM\rmap A^*$ on identity sections. As a result,
$\LF(\omega^\sharp)$ fits into the following commutative diagram:
$$
\xymatrix{ A(T\Grd)  \ar[r]^{\LF(\omega^\sharp)} \ar[d]_{} & A(T^*\Grd) \ar[d]^{} \\
TM \ar[r]_{\lambda} &  A^*, }
$$
and it follows from Proposition~\ref{prop:lie} that
$\Lambda=\LF(\omega)$ covers $\lambda$.

For part $(2)$, note that
$$
\LF(\sour^*\phi-\tar^*\phi)=\iota_A^*(\sour^*\phi)_T -
\iota_A^*(\tar^*\phi)_T.
$$
From \eqref{eq:magic} and the fact that $\mathrm{d}\phi=0$, we see that
$(\sour^*\phi)_T=\mathrm{d}\tau(\sour^*\phi)$ and
$(\tar^*\phi)_T=\mathrm{d}\tau(\tar^*\phi)$. A simple computation shows that
$\tau(\sour^*\phi)=(T\sour)^*\tau(\phi)$ and
$\tau(\tar^*\phi)=(T\tar)^*\tau(\phi)$. Hence
$\LF(\sour^*\phi-\tar^*\phi) = \mathrm{d}( \iota_A^* (T\sour)^*\tau(\phi) -
\iota_A^*(T\tar)^*\tau(\phi))$. Since $T\sour\circ \iota_A =0$ ($A$
is tangent to the $\sour$-fibres) and $T\tar \circ \iota_A=\rho$, we
obtain $\LF(\sour^*\phi-\tar^*\phi) = -\mathrm{d}\rho^*\tau(\phi)$. By
Corollary \ref{cor:lieunique}, we know that
$$
\mathrm{d}\omega - (\sour^*\phi-\tar^*\phi)=0 \iff \LF(\mathrm{d}\omega -
(\sour^*\phi-\tar^*\phi))=0.
$$
But $\LF(\mathrm{d}\omega - (\sour^*\phi-\tar^*\phi))= \mathrm{d}(\Lambda +
\rho^*\tau(\phi))$. Since the linear 2-form $\Lambda +
\rho^*\tau(\phi)$ covers $\lambda$, it follows from
Proposition~\ref{prop:linear} that
$$
\mathrm{d}(\Lambda + \rho^*\tau(\phi))=0 \iff \Lambda +
\rho^*\tau(\phi)=(\lambda^t)^*\omega_{can},
$$
as desired.
\end{proof}

To make the connection between this paper and the results in
\cite{bcwz} more transparent, it will be convenient to consider the
map $\sigma_\omega: A\rmap T^*M$ induced by $\omega\in
\Omega^2(\Grd)$ via
\begin{equation}\label{eq:sigma}
\sigma_\omega(u)(X)=\omega(u,X), \;\; u\in A, X\in TM.
\end{equation}
In the notation of Proposition~\ref{prop:mult}, we have
$\sigma_\omega=-\lambda^t$, so under the assumptions in part $(2)$,
$\Lambda=\LF(\omega)$ and $\sigma_\omega$ are related by
\begin{equation}\label{eq:relation}
\Lambda = -(\sigma_\omega^*\omega_{can} + \rho^*\tau(\phi)),
\end{equation}
in such a way that $\Lambda$ covers $-\sigma_\omega^t: TM \rmap
A^*$.

\subsection{IM 2-forms and Lie algebroid
morphisms}\label{subsec:IMmorphism} This subsection presents the key
step to the integration of IM 2-forms.

Let $A\rmap M$ be a Lie algebroid, with bracket $[\cdot,\cdot]$ and
anchor $\rho$. Let $\sigma:A\rmap T^*M$ be a vector bundle map (over
the identity) and $\phi\in \Omega^3(M)$ a closed 3-form. Motivated
by \eqref{eq:relation}, let us consider the linear 2-form $\Lambda
\in \Omega^2(A)$ defined by
\begin{equation}\label{eq:rel2}
\Lambda =  -(\sigma^*\omega_{can} + \rho^*\tau(\phi)),
\end{equation}
covering $-\sigma^t:TM\rmap A^*$. The following result describes
when such a 2-form induces a morphism between the tangent and
cotangent algebroid structures.

\begin{theorem}\label{thm:algmorp}
Let $\Lambda\in \Omega^2(A)$ be as in \eqref{eq:rel2}. The following
are equivalent:
\begin{itemize}
\item[(i)] The map $\Lambda^\sharp: TA \rmap T^*A$ is a Lie
algebroid morphism.
\item[(ii)] The map $\sigma:A\rmap T^*M$ satisfies
\begin{align}
\SP{\sigma(u),\rho(v)}&= - \SP{\sigma(v),\rho(u)}\label{eq:IM1}\\
\sigma([u,v])&=\Lie_{\rho(u)}\sigma(v)-i_{\rho(v)}\mathrm{d}\sigma(u) +
i_{\rho(v)}i_{\rho(u)}\phi, \label{eq:IM2}
\end{align}
for all $u,v \in \Gamma(A)$.
\end{itemize}
\end{theorem}
Vector bundle maps $\sigma:A \rmap T^*M$ satisfying conditions
\eqref{eq:IM1} and \eqref{eq:IM2} were introduced in \cite{bcwz} and
are referred to as \textbf{IM 2-forms} on $A$ (relative to $\phi$).
We also recall that a \textbf{morphism} between Lie algebroids
$A\rmap M$ and $B\rmap N$ (see, e.g., \cite{Mac-book}) is a vector
bundle map $\Psi:A\rmap B$, covering $\psi:M\rmap N$, which is
compatible with anchors, meaning that
$$
\rho_B \circ \Psi = T\psi \circ \rho_A,
$$
and compatible with brackets in the following sense. Consider the
pull-back bundle $\psi^*B\rmap M$, and let us keep denoting by
$\Psi$ the induced map $\Gamma(A)\rmap \Gamma(\psi^*B)$ at the level
of sections. Given sections $u,v\in \Gamma(A)$ such that
$\Psi(u)=f_j\psi^*u_j$ and $\Psi(v)=g_i\psi^*v_i$, where $f_j,g_i
\in C^\infty(M)$ and $u_j,v_i \in \Gamma(B)$, the following
condition should be valid:
\begin{equation}\label{eq:brackets}
\Psi([u,v]_{\st{A}})=f_jg_i\psi^*[u_j,v_i]_{\st{B}} +
\Lie_{\rho_A(u)}g_i \psi^*v_i - \Lie_{\rho_A(v)}f_j \psi^*u_j.
\end{equation}

We will need explicit local formulas for the tangent and cotangent
Lie algebroids. For a basis of local sections $\{e_d\}$ of $A$, we
denote the corresponding Lie algebroid structure functions by
$\rho^{j}_a$ and $C^{c}_{a b}$,
$$
\rho_A(e_a)= \rho_a^j \frac{\partial}{\partial x^j}, \;\;\;
[e_a,e_b] =C_{a b}^c e_c.
$$
Recall from Section \ref{subsec:tancot} that any section $u:M\rmap
A$ defines two types of sections of $TA\rmap TM$, denoted by $Tu$
and $\widehat{u}$. From \eqref{eq:TArel}, the tangent Lie algebroid
structure can be written as follows:


\begin{align}
&[\widehat{e}_a,\widehat{e}_b]_{\st{TA}}=0, \;\;
[T{e}_a,\widehat{e}_b]_{\st{TA}}=C^{c}_{ab}\widehat{e}_c, \;\;
[Te_a,Te_b]_{\st{TA}}= C_{ab}^cTe_c + \mathrm{d}C_{ab}^c\widehat{e}_c,\label{eq:TA1} \\
& \rho_{\st{TA}}(Te_a)=\rho^{j}_{a}\frac{\partial}{\partial x^j} +
\mathrm{d}\rho_a^j \frac{\partial}{\partial \dot{x}^j},\;\;\;
\rho_{\st{TA}}(\widehat{e}_a)=\rho^{j}_{a}\frac{\partial}{\partial
\dot{x}^j}.\label{eq:TA2}
\end{align}

To describe the Lie algebroid structure on $T^*A\rmap A^*$
explicitly, we also consider two types of sections that generate the
space of sections of $T^*A$ over $A^*$. The first type is induced
from a section $u\in \Gamma(A)$, and denoted by $u^L$. In local
coordinates $(x^j, \xi_d)$ on $A^*$ (relative to the basis of local
sections $\{e^d\}$ of $A^*$, dual to $\{ e_d\}$), it is given by
\begin{equation}\label{eq:linearT*}
u^L(x^j,\xi_d)=(x^j, u^d(x),0, \xi_d),
\end{equation}
where $T^*A$ is written locally in coordinates $(x^j, u^d, p_j,
\zeta_d)$ as in Section \ref{subsec:DVB}. The second type are
\textit{core sections}: locally, for each $\alpha=\alpha_j\mathrm{d}x^j \in
\Gamma(T^*M)$, we define the section $\widehat{\alpha}$ of
$T^*A\rmap A^*$ by
\begin{equation}\label{eq:coreT*}
\widehat{\alpha}(x^j,\xi_d)=(x^j,0,\alpha_j(x),\xi_d).
\end{equation}
The cotangent Lie algebroid is defined by the relations:

\begin{align}
&[\widehat{\mathrm{d}x^i},\widehat{\mathrm{d}x^j}]_{\st{T^*A}}=0, \;\;
[e_a^L,\widehat{\mathrm{d}x^j}]_{\st{T^*A}}=  \widehat{\mathrm{d}\rho_a^j}, \;\;
[e_a^L,e_b^L]_{\st{T^*A}}|_{(x,\xi)}=-\widehat{\mathrm{d}C^{c}_{ab}} \xi_c + C_{ab}^c e_c^L,\label{eq:T*A1}\\
&
\rho_{\st{T^*A}}(\widehat{\mathrm{d}x^i})=\rho^{i}_{a}\frac{\partial}{\partial
\xi_a},\;\;\;
\rho_{\st{T^*A}}(e_a^L)|_{(x,\xi)}=\rho^{i}_{a}\frac{\partial}{\partial
{x}^i} + C_{ab}^c \xi_c \frac{\partial}{\partial
\xi_b}.\label{eq:T*A2}
\end{align}

We now turn to the proof of Theorem~\ref{thm:algmorp}.

\begin{proof}
We work locally, so we assume that $M$ has coordinates $(x^j)$. Then
$A$ has coordinates $(x^j,u^d)$ (relative to a basis of local
sections $\{e_d\}$), $TA$ has tangent coordinates
$(x^j,u^d,\dot{x}^j,\dot{u}^d)$, while induced coordinates on $T^*A$
are denoted by $(x^j,u^d,p_j,\zeta_d)$. Similarly, $A^*$ has dual
coordinates $(x^j,\xi_d)$, inducing coordinates
$(x^j,\xi_d,\dot{x}^j,\dot{\xi}_d)$ on $TA^*$.

We start by discussing when $\Lambda^\sharp$ is compatible with the
anchors, i.e.,
\begin{equation}\label{eq:compatrho}
T(-\sigma^t) \circ \rho_{\st{TA}} = \rho_{\st{T^*A}}\circ
\Lambda^\sharp.
\end{equation}
Let us consider local expressions of the relevant maps. We write
$\sigma:A\rmap T^*M$ and $\sigma^t:TM\rmap A^*$ locally as
$$
\sigma(x^j,u^d)=(x^j,u^d\sigma_{jd}(x)),\;\;\;
\sigma^t(x^j,\dot{x}^j)=(x^j,\dot{x}^j\sigma_{jd}(x)).
$$
Denoting coordinates on $TM$ by $(x^j,\dot{x}^j)$, and on $T(TM)$ by
$(x^j,\dot{x}^j,\delta x^j, \delta \dot{x}^j)$, we get
$$
T(-\sigma^t)(x^j,\dot{x}^j,\delta x^j, \delta
\dot{x}^j)=(x^j,-\dot{x}^l\sigma_{ld},\delta {x}^j,
-\dot{x}^l\frac{\partial \sigma_{ld}}{\partial x^k} \delta x^k -
\sigma_{ld}\delta \dot{x}^l) \in TA^*.
$$
One can directly verify that the map $\Lambda^\sharp$ can be locally
written as follows:
\begin{equation}\label{eq:localL}
\Lambda^\sharp(x^j,u^d,\dot{x}^j,\dot{u}^d)= (x^j,u^d, p_j ,
\zeta_d),
\end{equation}
where
$$
p_j= \dot{x}^lu^d \left (\frac{\partial \sigma_{jd}}{\partial x^l} -
\frac{\partial \sigma_{ld}}{\partial x^j} \right ) +
\dot{u}^d\sigma_{jd}-\phi_{ijk}u^d\rho_d^k\dot{x}^i, \; \;\;
\zeta_d= -\dot{x}^l\sigma_{ld}.
$$
The space of sections of $TA \rmap TM$ is generated by sections of
types $Te_a$ and $\widehat{e}_b$. We have
\begin{align}
\Lambda^\sharp(Te_a|_{(x,\dot{x})})&= \left(x^j,\delta_{ad},
\dot{x}^l \left(\frac{\partial \sigma_{ja}}{\partial x^l}-
\frac{\partial \sigma_{la}}{\partial x^j}\right) - \phi_{ijk}\rho_a^k \dot{x}^i, -\dot{x}^l\sigma_{ld}\right),\label{eq:L1}\\
\Lambda^\sharp(\widehat{e}_b|_{(x,\dot{x})})& = (x^j, 0,
\sigma_{jb}, -\dot{x}^l\sigma_{ld}).\label{eq:L2}
\end{align}

Using \eqref{eq:TA2} and \eqref{eq:T*A2}, one can directly check
that
$$
T(-\sigma^t) (\rho_{\st{TA}}(\widehat{e}_b|_{(x,\dot{x})}))= (x^j,
-\dot{x}^l\sigma_{ld},0,-\sigma_{ld}\rho_b^l) \in (-\sigma^t)^*TA^*.
$$
On the other hand, using the local expression \eqref{eq:localL}, we
have
$$
\rho_{\st{T^*A}}(\Lambda^\sharp(\widehat{e}_b|_{(x,\dot{x})}))=
(x^j,-\dot{x}^l\sigma_{ld},0,\rho^l_d\sigma_{lb}).
$$
It follows that the compatibility \eqref{eq:compatrho} for core
sections amounts to
$$
\SP{\rho(e_b),\sigma(e_d)}=-\SP{\rho(e_d),\sigma(e_b))},
$$
which is equivalent to \eqref{eq:IM1}.

For sections of type $Te_b$, again using \eqref{eq:TA2} and
\eqref{eq:T*A2}, we get
$$
T(-\sigma^t)(\rho_{\st{TA}}(Te_b|_{(x,\dot{x})}))=(x^j, -\dot{x}^l
\sigma_{ld},\rho^j_b,\zeta_d) \in (-\sigma^t)^*TA^*,
$$
where
\begin{equation}\label{eq:zeta1}
\zeta_d = -\dot{x}^l\left (\frac{\partial \sigma_{ld}}{\partial
x^k}\rho^k_b + \sigma_{id}\frac{\partial \rho^i_b}{\partial
x^l}\right) =  -\SP{\Lie_{\rho(e_b)}\sigma(e_d),\dot{x}}.
\end{equation}
Similarly, we compute
$$
\rho_{\st{T^*A}}(\Lambda^\sharp(Te_b|_{(x,\dot{x})}))= (x^j,
-\dot{x}^l \sigma_{ld},  \rho_b^j, \zeta'_d),
$$
where
\begin{align}
\zeta'_d &= \dot{x}^l\rho_d^k \left(\frac{\partial
\sigma_{kb}}{\partial x^l}  -\frac{\partial \sigma_{lb}}{\partial
x^k} \right ) - \phi_{ijk}\rho_b^k\dot{x}^i\rho_d^j +
C_{db}^c\dot{x}^l\sigma_{lc}\nonumber\\
&=\SP{-i_{\rho(e_d)} (\mathrm{d}\sigma(e_b)) + i_{\rho(e_d)}i_{\rho(e_b)}\phi
+ \sigma([e_d,e_b]),\dot{x}}.\label{eq:zeta2}
\end{align}
Comparing \eqref{eq:zeta1} and \eqref{eq:zeta2}, it follows that the
compatibility \eqref{eq:compatrho} for sections of the type $Te_b$
is verified if and only if \eqref{eq:IM2} holds.

Let us now check the bracket-preserving condition
\eqref{eq:brackets}, that in our case reads
\begin{align}\label{eq:brkpres}
\Lambda^\sharp([U,V]_{\st{TA}}|_{(x,\dot{x})})  = & f_j g_i
[U_j,V_i]_{\st{T^*A}}|_{-\sigma^t(x,\dot{x})} +
\Lie_{\rho_{TA}(U)}g_i V_i|_{-\sigma^t(x,\dot{x})}\\
& - \Lie_{\rho_{TA}(V)}f_j U_j|_{-\sigma^t(x.\dot{x})},\nonumber
\end{align}
where $U,V\in \Gamma(TA)$, and $f_j,g_i \in C^\infty(TM)$, $U_j,V_i
\in \Gamma(T^*A)$ are such that $\Lambda^\sharp(U)=f_j
(-\sigma^t)^*U_j$ and $\Lambda^\sharp(V)=g_i (-\sigma^t)^*V_i$.

From \eqref{eq:L1}, \eqref{eq:L2}, we can write
\begin{align}
\Lambda^\sharp(Te_a|_{(x,\dot{x})})&= e^L_a|_{-\sigma^t(x,\dot{x})}
+ f^a_j \widehat{\mathrm{d}x^j}|_{-\sigma^t(x,\dot{x})}, \label{eq:Lf}\\
\Lambda^\sharp(\widehat{e}_a|_{(x,\dot{x})})& =
\widehat{\sigma(e_a)}|_{-\sigma^t(x,\dot{x})} = g^a_i
\widehat{\mathrm{d}x^i}|_{-\sigma^t(x,\dot{x})},\label{eq:Lg}
\end{align}
where
\begin{equation}\label{eq:fg}
f^a_j = \dot{x}^l \left(\frac{\partial \sigma_{ja}}{\partial x^l}-
\frac{\partial \sigma_{la}}{\partial x^j}\right) -
\phi_{ijk}\rho_a^k \dot{x}^i, \;\;\; g^a_i = \sigma_{ia},
\end{equation}
so we can express the images in terms of sections of types
\eqref{eq:linearT*} and \eqref{eq:coreT*} on $T^*A$. It will be
useful to note that the functions $f^a_j=f^a_j(x,\dot{x})$ satisfy
\begin{equation}\label{eq:feq}
f^a_j\mathrm{d}x^j=i_{\dot{x}}\mathrm{d}\sigma(e_a) - i_{\dot{x}}i_{\rho(e_a)}\phi,
\end{equation}
viewed as an equality of \textit{horizontal} 1-forms on $TM$, i.e.
1-forms of type $\alpha_j(x,\dot{x})\mathrm{d}x^j$ (in this formula,
$\dot{x}$ is seen as the vector field
$\dot{x}^l\frac{\partial}{\partial x^l}$ on $TM$). In fact, locally,
there is an identification of the space of horizontal 1-forms on
$TM$ with a subspace of sections of $(-\sigma^t)^*T^*A$ via
\begin{equation}
\Omega^1_{hor}(TM) \rmap \Gamma((-\sigma^t)^*T^*A),\;\;\;\;
\alpha_j(x,\dot{x})\mathrm{d}x^j\mapsto
\alpha_j(x,\dot{x})\widehat{\mathrm{d}x^j}|_{-\sigma^t(x,\dot{x})}.
\end{equation}
In the remainder  of this section, we will use this identification
to view horizontal 1-forms on $TM$ as sections of the bundle
$(-\sigma^t)^*T^*A$. In particular, in order to simplify our
notation, we will write $\widehat{\mathrm{d}x^j}|_{-\sigma^t(x,\dot{x})}$
just as $\mathrm{d}x^j$.

Since it suffices to verify condition \eqref{eq:brkpres} for
sections of types $Te_a$ (linear) and $\widehat{e}_a$ (core), we
have three cases to analyze.

\medskip

\noindent{\it Core-core sections}

\smallskip

If $U=\widehat{e}_a$ and $V=\widehat{e}_b$ are core sections, then
by \eqref{eq:TA1} we know that
$[\widehat{e}_a,\widehat{e}_b]_{\st{TA}}=0$, so the left-hand side
of \eqref{eq:brkpres} vanishes. On the other hand, from
\eqref{eq:TA2}, the Lie derivatives on the right-hand side. of
\eqref{eq:brkpres} are only with respect to the variable $\dot{x}$.
Since the functions $g_j$ in \eqref{eq:Lg} do not depend on
$\dot{x}$ and $[\widehat{\mathrm{d}x^i},\widehat{\mathrm{d}x^j}]_{\st{T^*A}}=0$, it
follows that, for a pair of core sections, the right-hand side of
\eqref{eq:brkpres} vanishes as well.

\medskip

\noindent{\it Core-linear sections}

\smallskip

Let us consider \eqref{eq:brkpres} when $U=Te_a$ and
$V=\widehat{e}_b$. Since $[Te_a,\widehat{e}_b]_{\st{TA}}= C_{ab}^c
\widehat{e}_c$, it follows from \eqref{eq:Lg} that the left-hand
side of \eqref{eq:brkpres} is
$$
\Lambda^\sharp([Te_a,\widehat{e}_b]_{\st{TA}})= \sigma([e_a,e_b]).
$$

Using the bracket relations \eqref{eq:T*A1}, one directly sees that
the first term on the right-hand side of \eqref{eq:brkpres} is just
$\sigma_{ib}\mathrm{d}\rho_{a}^i$. For the second term, we have
\begin{align*}
(\Lie_{\rho_{\st{TA}}(Te_a)}\sigma_{ib}) \mathrm{d}x^i & =
\Lie_{\rho_a^l\frac{\partial}{\partial x^l}}(\sigma_{ib}\mathrm{d}x^i)-
\sigma_{ib}\Lie_{\rho_a^l\frac{\partial}{\partial
x^l}}\mathrm{d}x^i \\
&= \Lie_{\rho(e_a)}\sigma(e_b)- \sigma_{ib} \mathrm{d}\rho_a^i.
\end{align*}
The third term on the right-hand side of \eqref{eq:brkpres} is given
by
\begin{align*}
(\Lie_{\rho_{\st{TA}}(\widehat{e}_b)} f^a_j) \mathrm{d}x^j & = \left
(\rho_b^l\left(\frac{\partial\sigma_{ja}}{\partial x^l}-
\frac{\partial \sigma_{la}}{\partial x^j}\right) -
\phi_{ijk}\rho_a^k \rho_b^i \right )\mathrm{d}x^j\\
&= i_{\rho(e_b)}\mathrm{d}\sigma(e_a) - i_{\rho(e_b)}i_{\rho(e_a)}\phi.
\end{align*}
As a result, in this case, \eqref{eq:brkpres} is equivalent to
$$
\sigma([e_a,e_b]) = \Lie_{\rho(e_a)}\sigma(e_b) -
i_{\rho(e_b)}\mathrm{d}\sigma(e_a) + i_{\rho(e_b)}i_{\rho(e_a)}\phi,
$$
which agrees with condition \eqref{eq:IM2}.

\medskip

\noindent{\it Linear-linear sections}

\smallskip

We finally consider \eqref{eq:brkpres} when $U=Te_a$ and $V=Te_b$.
From \eqref{eq:TA1}, \eqref{eq:Lf} and \eqref{eq:Lg}, and using
\eqref{eq:feq}, we see that the left-hand side of \eqref{eq:brkpres}
is
\begin{align*}
\Lambda^\sharp([Te_a,Te_b]_{\st{TA}})&= C_{ab}^c
e_c^L|_{-\sigma^t(x,\dot{x})} + C_{ab}^cf_j^c\mathrm{d}x^j+
\mathrm{d}C_{ab}^c(\dot{x})\sigma(e_c)\\
&= [e_a,e_b]^L|_{-\sigma^t(x,\dot{x})} +
C_{ab}^c(i_{\dot{x}}\mathrm{d}\sigma(e_c) - i_{\dot{x}}i_{\rho(e_c)}\phi) +
\mathrm{d}C_{ab}^c(\dot{x})\sigma(e_c)\\
&= [e_a,e_b]^L|_{-\sigma^t(x,\dot{x})} +
i_{\dot{x}}\mathrm{d}\sigma([e_a,e_b])+ \mathrm{d}C_{ab}^c \SP{\sigma(e_c),\dot{x}}
-i_{\dot{x}}i_{[\rho(e_a),\rho(e_b)]}\phi.
\end{align*}
As in \eqref{eq:feq}, we abuse notation and use $\dot{x}$ also to
represent the vector field $\dot{x}^l\frac{\partial}{\partial x^l}$.

 By \eqref{eq:Lf}, we can write
$\Lambda^\sharp(Te_a)=e_a^L+f_j^a \mathrm{d}x^j$ and
$\Lambda^\sharp(Te_b)=e_b^L+f_i^b \mathrm{d}x^i$. Using \eqref{eq:T*A1}, we
see that the first term on the right-hand side of \eqref{eq:brkpres}
is
$$
[e_a,e_b]^L|_{-\sigma^t(x,\dot{x})} +
\mathrm{d}C_{ab}^c\SP{\sigma^t(\dot{x}),e_c} - f_j^a {\mathrm{d}\rho_b^j} +
f_i^b{\mathrm{d}\rho_a^i}.
$$
Note that the second and third terms on the right-hand side of
\eqref{eq:brkpres} are $\Lie_{\rho_{\st{TA}}(Te_a)}f^b_i \mathrm{d}x^i$ and
$\Lie_{\rho_{\st{TA}}(Te_b)}f^a_j {\mathrm{d}x^j}$, respectively. Let us find
a more explicit expression for the latter (the former is clearly
completely analogous). Since
$$
\Lie_{\rho_{\st{TA}}(Te_b)}f^a_j {\mathrm{d}x^j}=
\Lie_{\rho_{\st{TA}}(Te_b)}(f^a_j {\mathrm{d}x^j})-f^a_j
{\Lie_{\rho_{\st{TA}}(Te_b)}\mathrm{d}x^j},
$$
it follows that
$$
\Lie_{\rho_{\st{TA}}(Te_b)}f^a_j
{\mathrm{d}x^j}={\Lie_{\rho_{\st{TA}}(Te_b)}i_{\dot{x}}\mathrm{d}\sigma(e_a)} -
{\Lie_{\rho_{\st{TA}}(Te_b)}i_{\dot{x}}i_{\rho(e_a)}\phi}-f^a_j{\mathrm{d}\rho_b^j}.
$$
Let us consider the (local) vector fields on $TM$ given by $V_b=
\mathrm{d}\rho^l_b(\dot{x})\frac{\partial}{\partial {x}^l}$ and
$V'_b=\mathrm{d}\rho^l_b(\dot{x})\frac{\partial}{\partial \dot{x}^l}$, so
that $\rho_{\st{TA}}(Te_b)=\rho(e_b) + V'_b$. It is simple to check
that $[\rho(e_a), \dot{x}]=-V_b$ and $\Lie_{V_b'}i_{\dot{x}}\alpha=
i_{V_b}\alpha$ for any $\alpha=\frac{1}{2}\alpha_{ij}(x)\mathrm{d}x^i\wedge
\mathrm{d}x^j$. Using Cartan calculus, we find
\begin{align*}
\Lie_{\rho_{\st{TA}}(Te_b)}i_{\dot{x}}\mathrm{d}\sigma(e_a)& =
\Lie_{\rho(e_b)}i_{\dot{x}}\mathrm{d}\sigma(e_a) +
\Lie_{V_b'}i_{\dot{x}}\mathrm{d}\sigma(e_a)\\
& = -i_{V_b}\mathrm{d}\sigma(e_a) + i_{\dot{x}}\Lie_{\rho(e_b)}\mathrm{d}\sigma(e_a) +
i_{V_b}\mathrm{d}\sigma(e_a)\\
&=i_{\dot{x}}di_{\rho(e_b)}\mathrm{d}\sigma(e_a).
\end{align*}
Similarly,
$$
\Lie_{\rho_{\st{TA}}(Te_b)}i_{\dot{x}}i_{\rho(e_a)}\phi=
i_{\dot{x}}\Lie_{\rho(e_b)}i_{\rho(e_a)}\phi.
$$
As a result, we obtain
$$
\Lie_{\rho_{\st{TA}}(Te_b)}f^a_j {\mathrm{d}x^j}=
i_{\dot{x}}di_{\rho(e_b)}\mathrm{d}\sigma(e_a)-
i_{\dot{x}}\Lie_{\rho(e_b)}i_{\rho(e_a)}\phi -f^a_j {\mathrm{d}\rho_b^j}.
$$
Analogously, we have
$$
\Lie_{\rho_{\st{TA}}(Te_a)}f^b_i \mathrm{d}x^i  =
i_{\dot{x}}di_{\rho(e_a)}\mathrm{d}\sigma(e_b)-
i_{\dot{x}}\Lie_{\rho(e_a)}i_{\rho(e_b)}\phi -f^b_i{\mathrm{d}\rho_a^i}.
$$
Hence \eqref{eq:brkpres} amounts to the identity
\begin{align}\label{eq:lastid}
{i_{\dot{x}}\mathrm{d}\sigma([e_a,e_b])} -
{i_{\dot{x}}i_{[\rho(e_a),\rho(e_b)]}\phi}=&
i_{\dot{x}}\mathrm{d}i_{\rho(e_a)}\mathrm{d}\sigma(e_b)-
i_{\dot{x}}\Lie_{\rho(e_a)}i_{\rho(e_b)}\phi\\ &-
i_{\dot{x}}\mathrm{d}i_{\rho(e_b)}\mathrm{d}\sigma(e_a)+
i_{\dot{x}}\Lie_{\rho(e_b)}i_{\rho(e_a)}\phi.\nonumber
\end{align}
By basic Cartan calculus of forms, we have the identity
$$
i_{[\rho(e_a),\rho(e_b)]}\phi-\Lie_{\rho(e_a)}i_{\rho(e_b)}\phi+\Lie_{\rho(e_b)}i_{\rho(e_a)}\phi
= \mathrm{d}i_{\rho(e_b)}i_{\rho(e_a)}\phi.
$$
It follows that \eqref{eq:lastid} is equivalent to
\begin{align*}
\mathrm{d}\sigma([e_a,e_b])&=\mathrm{d}(i_{\rho(e_a)}\mathrm{d}\sigma(e_b)-i_{\rho(e_b)}\mathrm{d}\sigma(e_a)
+ i_{\rho(e_b)}i_{\rho(e_a)}\phi)\\
&= \mathrm{d}(\Lie_{\rho(e_a)}\sigma(e_b)
-\mathrm{d}i_{\rho(e_a)}\sigma(e_b)-i_{\rho(e_b)}\mathrm{d}\sigma(e_a) +
i_{\rho(e_b)}i_{\rho(e_a)}\phi)\\
&= \mathrm{d}(\Lie_{\rho(e_a)}\sigma(e_b)-i_{\rho(e_b)}\mathrm{d}\sigma(e_a) +
i_{\rho(e_b)}i_{\rho(e_a)}\phi),
\end{align*}
which holds by \eqref{eq:IM2}.
\end{proof}
\subsection{Applications to integration}

In this section we present an alternative proof of the main result
in \cite{bcwz}, which describes IM 2-forms as infinitesimal versions
of multiplicative 2-forms (\cite[Theorem~2.5]{bcwz}).

Let $\Grd$ be a Lie groupoid over $M$, with Lie algebroid $A$. Let
us denote the space of multiplicative 2-forms on $\Grd$ by
$\Omega^2_{mult}(\Grd)$, and the space of linear 2-forms on $A$ by
$\Omega^2_{lin}(A)$. We also consider the subspace
$\Omega^2_{alg}(A)\subset \Omega^2_{lin}(A)$ of linear 2-forms
$\Lambda$ for which $\Lambda^\sharp:TA\rmap T^*A$ is a Lie algebroid
morphism.

A direct consequence of Proposition~\ref{prop:lie} is that
$\LF(\alpha)^\sharp$ is a Lie algebroid morphism for any
multiplicative $k$-form $\alpha$ on $\Grd$. Using
Proposition~\ref{prop:mult}, part $(1)$, we conclude that the Lie
functor on multiplicative forms gives rise to a well-defined map
\begin{equation}\label{eq:LF}
\LF: \Omega^2_{mult}(\Grd)\rmap \Omega^2_{alg}(A), \;\; \omega
\mapsto \Lambda=\LF(\omega).
\end{equation}

\begin{proposition}
If $\Grd$ is $\sour$-simply-connected, then \eqref{eq:LF} is a
bijection.
\end{proposition}

\begin{proof}
We will show that \eqref{eq:LF} has an inverse map. If
$\Lambda\in\Omega^2_{alg}(A)$, then $\Lambda^\sharp:TA\rmap T^*A$ is
a morphism of algebroids. So
$$
\theta_\Grd^{-1} \circ  \Lambda^\sharp \circ j_\Grd^{-1}:
A(T\Grd)\rmap A(T^*\Grd)
$$
is a Lie algebroid morphism. Since $\Grd$ is
$\sour$-simply-connected, so is $T\Grd$. By Lie's second theorem for
algebroids (see, e.g., \cite{Mac-book}), there exists a unique Lie
groupoid morphism $\omega^\sharp: T\Grd \rmap T^*\Grd$ with
$\LF(\omega^\sharp)=\theta_\Grd^{-1} \circ  \Lambda^\sharp \circ
j_\Grd^{-1}$, or $\Lambda^\sharp = \theta_\Grd\circ
\LF(\omega^\sharp) \circ j_\Grd$.

It remains to check that $\omega^\sharp$ is indeed the bundle map
associated with a 2-form $\omega \in \Omega^2(\Grd)$, i.e., that it
is a vector-bundle map (covering the identity) with respect to the
bundle structures $T\Grd\rmap \Grd$ and $T^*\Grd \rmap \Grd$, and
that $(\omega^\sharp)^t = -\omega^\sharp$. A proof of this fact can
be given just as in \cite{Mac-Xu2}: the key point is that the bundle
projections $p_\Grd : T\Grd \rmap \Grd$, $c_\Grd:T^*\Grd\rmap \Grd$,
the vector bundle sums $T\Grd\times_{p_\Grd} T\Grd \rmap T\Grd$,
$T^*\Grd\times_{c_\Grd} T^*\Grd \rmap T^*\Grd$, and scalar
multiplications $T\Grd \times \mathbb{R}\rmap T\Grd$, $T^*\Grd
\times \mathbb{R}\rmap T^*\Grd$, as well as the natural pairing
$T\Grd\times_{(p_\Grd,c_\Grd)} T^*\Grd\rmap \mathbb{R}$ are all
groupoid morphisms. The corresponding maps for Lie algebroids (after
the identifications \eqref{eq:j} and \eqref{eq:thetag}) are
precisely the vector bundle structure maps and pairing for $p_A:
TA\rmap A$ and $c_A: T^*A\rmap A$, see, e.g., \cite{Mac-Xu}. For
example, to prove that $c_\Grd\circ \omega^\sharp=p_\Grd$, it
suffices to verify this condition (by the connectivity of the
source-fibres) at the level of algebroids. But then we have
$$
\LF(c_\Grd\circ \omega^\sharp)= c_A\circ \Lambda^\sharp =
p_A=\LF(p_\Grd).
$$
The other properties of $\omega^\sharp$ are derived from those of
$\Lambda^\sharp$ similarly, as in \cite[Theorem~4.1]{Mac-Xu2}.
\end{proof}

\begin{corollary}[\cite{bcwz}]
If $\Grd$ is $\sour$-simply-connected and $\phi\in\Omega^3(M)$ is
closed, there is a one-to-one correspondence between multiplicative
2-forms on $\Grd$ satisfying $\mathrm{d}\omega = \sour^*\phi-\tar^*\phi$ and
IM 2-forms $\sigma:A\rmap T^*M$ relative to $\phi$.
\end{corollary}

\begin{proof}
We know that $\LF:\Omega^2_{mult}(\Grd)\rmap \Omega^2_{alg}(A)$,
$\omega\mapsto \Lambda=\LF(\omega)$ is a bijection, and by
Proposition~\ref{prop:mult}, part $(2)$,
$\mathrm{d}\omega=\sour^*\phi-\tar^*\phi$ if and only if
$\Lambda=-(\sigma_\omega^*\omega_{can} + \rho^*\tau(\phi))$. The
conclusion now follows from Theorem~\ref{thm:algmorp}.
\end{proof}


\end{document}